\documentclass[english, 12pt]{article}

\usepackage{amsfonts}
\usepackage{amsmath}
\usepackage{amssymb}
\usepackage{amsthm}
\usepackage{amscd}
\usepackage{amscd}
\usepackage{amsthm}

\newtheorem{theo}{Theorem}[section]
\newtheorem{prop}{Proposition}[section]
\newtheorem{lem}{Lemma}[section]

\newtheorem{rk}{Remark}[section]

\newtheorem{exa}{Example}[section]
\newtheorem{coro}{Corollary}[section]
\numberwithin{equation}{section}

\def\R{\mathbb{R}}

\def\calE{{\cal{E}}}
\def\calF{{\cal{F}}}

\newcommand{\varp}{\varphi}
\newcommand{\lam}{\lambda}
\newcommand{\Lam}{\Lambda}
\newcommand{\Om}{\Omega}

\newcommand{\alp}{\alpha}

\begin{document}
\bibliographystyle{alpha}

\title{Hardy's inequality in the scope of Dirichlet forms}

\author{\normalsize Nedra Belhadjrhouma\footnote {Facult\'e des Sciences de Tunis, Department of Mathematics. Tunisia }
\& Ali BenAmor\footnote{corresponding author} \footnote{Department of Mathematics, Uni.Gafsa, Tunisia. E-mail: ali.benamor@ipeit.rnu.tn} }

\date{}
\maketitle
\begin{abstract} We revisit Hardy's inequality in the scope of regular Dirichlet forms
following an analytical method. We shall give an alternative necessary and sufficient
condition for the occurrence of Hardy's inequality. A special emphasis
will be given for the case where
the Dirichlet form under consideration is strongly local, extending therefore some known results
in the Euclidean case.
\end{abstract}
\section{Introduction}
By Hardy's inequality we mean an inequality of the type
\begin{eqnarray}
\int f^2\,d\mu\leq C\calE[f],\ \forall\,f\in D(\calE),
\label{hardy inequality}
\end{eqnarray}
where $\calE$ is a Dirichlet form and $\mu$ a positive measure charging no set having zero capacity.\\
Discussions of such type of inequalities and its consequences in the scope of Dirichlet forms
was made by several authors and the subject has gained much more interest in the last years
\cite{kaimanovich,vondracek,fitz,fuku02,benamor04,rao}, due to their relevance to many areas of
mathematics (spectral theory, PDE's, potential theory,...etc).\\
In the literature there are many type of necessary and sufficient conditions for the validity
of inequality(\ref{hardy inequality}) (especially for the gradient energy form on Euclidean
domains (\cite{adams73,ancona,mazja85,adams,tidblom}): capacitary conditions, functional
conditions...etc\\
In \cite{ancona}, Ancona proved that Hardy's inequality holds true on Euclidean domains $\Om$ for
the measure $({\rm dist}(x,\partial\Om))^{-2}dx$ and where the energy is the gradient energy form
if and only if $\Om$ possesses a 'strong barrier'. Years after Fitzsimmons \cite{fitz}, proved
that this deep result holds true in a very large generality, namely for quasi-regula Dirichlet
form.\\
Being inspired by the papers of Ancona \cite{ancona} and Fitzsimmons \cite{fitz} we shall give new
 necessary and sufficient condition ensuring the validity of Hardy's inequality. In fact, using
 Beurling-Deny formula, we shall write Ancona's condition in a variational form, without assuming
 the barrier to be superharmonic.\\
 We shall also show that our condition is equivalent to the one due Fitzsimmons.\\ In the special
 (but relevant) case where the Dirichlet form is strongly local (of diffusion type), using the
 intrinsic metric induced by $\calE$, we shall generalize and improve the known Hardy \cite{ancona}
 and improved Hardy inequality on bounded Euclidean domains \cite{filippas} in our general setting.\\
 Our method is rather direct and analytic. It is based upon the use of the celebrated Beurling-Deny
 formula\\

 \section{Preliminaries}

We first shortly describe the framework in which we shall state our results.\\
Let $\calE$ be a regular symmetric transient Dirichlet form, with domain $\calF:=D(\calE)$ w.r.t.
the space $L^2:=L^2(X,m)$. We assume that $X$ is a  separable metric space and that
$m$ is a reference measure.\\
In this stage we would like to emphasize that our assumptions on the Dirichlet form are not very restrictive. Indeed, every {\em quasi-regular} Dirichlet form is {\em quasi-homeomorphic} to a {\em regular} Dirichlet form \cite{roeckner-homeo}. So that our results are true for quasi-regular Dirichlet forms as well.\\
The local Dirichlet space related to $\calE$ will be denoted by $\calF_{\rm loc}$.
A function $f$ belongs to $\calF_{\rm loc}$ if for every relatively compact subset $\Om\subset X$
there is $\tilde f\in D(\calE)$ such that $f=\tilde f$-a.e. on $\Om$.\\
We recall the known fact that every element from $\calF_{\rm loc}$ has a quasi-continuous modification. We shall always implicitly assume that elements from $\calF_{\rm loc}$ has been modified so as to become quasi-continuous.\\
We also designate by $\calF_b:=\calF\cap L^{\infty}(X,m)$ and $\calF_{b,\rm loc}:=\calF_{\rm loc}\cap L_{\rm loc}^{\infty}(X,m)$.\\
We shall denote respectively by $\kappa$, $J$ the killing and the jumping measures related to $\calE$ and $\calE^{c}$ its strong part both  given by Beurling-Deny formula (See \cite[Theorem 4.5.2,p.164]{fuku-oshima} (for quasi-regular Dirichlet forms, see \cite{kuwae}).\\

\begin{eqnarray*}
J(f,g)=\int_{X\setminus d\times X\setminus d}(f(x)-f(y))(g(x)-g(y))J(dx,dy),\ \forall\,f,g\in \calF.
\end{eqnarray*}

Given $f,g\in\calF$ set $\mu^c_{<f>}$ the {\em energy measure} of $f$ and  $\mu^c_{<f,g>}$ the {\em  mutual energy measure} of $f,g$ (see \cite[pp.110-114]{fuku-oshima}). Furthermore the strong local part of $\calE$ possesses the representation
\begin{eqnarray}
\calE^c[f]=\frac{1}{2}\mu^c_{<f>}(X),\ \forall\,f\in\calF.
\label{energy}
\end{eqnarray}
The representation goes as follows: for $f\in\calF_b$ its energy measure is defined by
\begin{eqnarray}
\int\phi\,d\mu^c_{<f>}=2\calE(f,\phi f)-\calE(f^2,\phi),\ \forall 0\leq\phi\in\calF\cap C_c(X).
\end{eqnarray}
Truncation and monotone convergence allow then to define $\mu^c_{<f>}$ for every $f\in\calF$.\\
Furthermore with the help of strong locality
\begin{eqnarray}
\int_{\{f=c\}}d\mu^c_{<f>}=0,\ \forall\,f\in\calF,
\end{eqnarray}
it is possible to define $\mu^c_{<f>}$ for every $f\in\calF_{\rm loc}$ as follows: for every relatively compact subset $\Om\subset X$
\begin{eqnarray}
1_{\Om}\mu^c_{<f>}=1_{\Om}\mu^c_{<\tilde f>},
\end{eqnarray}
where $\tilde f\in\calF$ and $f=\tilde f$-q.e. on $\Om$.\\
By polarization and regularity we can thereby define a Radon-measures-valued bilinear form on $\calF_{\rm loc}$ so that
\begin{eqnarray}
\calE^c(f,g)=\frac{1}{2}\mu^c_{<f,g>}(X),\ \forall\,f\in\calF_{\rm loc}.
\label{energy'}
\end{eqnarray}
The truncation property for $\calE^c$ reads as follows: For every $a\in\R$, every $f\in\calF_{\rm loc}$ and every $g\in\calF_{\rm b,loc}$ we have
\begin{eqnarray}
\calE^c((f-a)_+,g)=1_{\{f>a\}}\calE^c(f,g)\ {\rm and}\  \calE^c[(f-a)_+]=1_{\{f>a\}}\calE^c[f].
\label{truncation}
\end{eqnarray}
Furthermore the following product formula holds true
\begin{eqnarray}
d\mu^c_{<fh,g>}=fd\mu^c_{<h,g>}+hd\mu^c_{<f,g>},\ \forall\,f,g,h\in \calF_{\rm b,loc}.
\label{product}
\end{eqnarray}
By the regularity assumption the latter formula extends to every $f,g,h\in\calF_{\rm loc}$.\\ Another rule that we shall occasionally use is the {\em chain rule} (See \cite[pp.11-117]{fuku-oshima}): For every function $\phi:\R\to\R$ of class $C^1$ with bounded derivative ($\phi\in C_b^1(\R)$), every $f\in \calF_{\rm loc}$ and every $g\in\calF_{\rm b,loc}$, $\phi(f)\in\calF_{\rm loc}$ and
\begin{eqnarray}
d\mu^c_{<\phi(f),g>}=\phi'(f)d\mu^c_{<f,g>}.
\label{chain rule}
\end{eqnarray}
We improve a bit the chain rule.
\begin{lem} Let $\phi:(0,\infty)\to\R$ of class $C^1$ be such that for every $a>0$, $f\in C_b^1([a,\infty))$. Let $f\in\calF_{\rm loc}$ such that for every kompact subset $K\subset X$, there is $C_K>0$ such that $f\geq C_K$-q.e. on $K$. Then $\phi(f)\in\calF_{\rm loc}$ and
\begin{eqnarray}
d\mu^c_{<\phi(f),g>}=\phi'(f)d\mu^c_{<f,g>},\ \forall\,g\in\calF_{\rm b,loc}
\label{chainrule'}
\end{eqnarray}
\label{improve}
\end{lem}

\begin{proof} Let $K\subset X$, kompact and $C_K>0$ as in the lemma. Let $\tilde f\in\calF$ s.t. $\tilde f=f$-q.e.on $K$. We extend the restriction of $\phi$ to $[C_K,\infty)$ by a function $\tilde\phi\in C_b^1(\R)$. Then $\tilde\phi(\tilde f)\in\calF$ and $\tilde\phi(\tilde f)=\phi(f)$-q.e.on $K$ and by formula(\ref{chain rule})
\begin{eqnarray}
1_Kd\mu^c_{<\tilde\phi(\tilde f),g>}&=&1_K{\tilde\phi}'(f)d\mu^c_{<f,g>}\nonumber\\
&&= 1_Kd\mu^c_{<\phi(f),g>}=1_K\phi'(f)d\mu^c_{<f,g>}\forall\,g\in\calF_{\rm b,loc}
\end{eqnarray}
which was to be proved.
\end{proof}

We shall also make use of the following fact.

\begin{lem} Let $w$ be a q.c. function such that $w>0$-q.e. Then $w^{-1}$ is locally quasi-bounded.
\label{quasi-bounded}
\end{lem}
\begin{proof} By \cite[Theorem2.1.2]{fuku-oshima}, there is a nest $(F_k)$ s.t. $w\in C(F_k)$, for every $k$. Set $Y:=\cup_{k}F_k$, then $X\setminus Y$ has zero capacity.\\
For every integer $k$, we set
\begin{eqnarray}
G_k:=F_k\cap\{w\geq\frac{1}{k}\}.
\end{eqnarray}
Then $G_k$ is closed as well as for the topology of $X$ and that of $Y$ inherited from $X$. Also $K':=K\cap Y$ is compact w.r.t. to the trace topology of $X$ on $Y$. Since $(G_k)$ is a covering for $Y$ of closed sets, there is a finite number of $G_k$'s s.t. $K=\cup_{\rm finite}G_k$. Thus $\inf_{K'}w>0$.\\
On the other hand ${\rm Cap}(K\cap Y^c)\leq{\rm Cap}(X\setminus Y)=0$, yielding
\begin{eqnarray}
w(x)\geq\inf_{K'}w>0-{\rm q.e.\ on}\ K,
\end{eqnarray}
which was to be proved.
\end{proof}

\section{Hardy's inequality}

We are in position now to assert the first part of the main result.
\begin{theo} Let $\calE$ be a transient Dirichlet form and  $\mu$ be a positive Radon measure on Borel subsets of $X$,charging no sets having zero capacity. Assume that there is $C>0$ and a function $w\in\calF_{\rm loc}$, $w>0$-q.e., such that
\begin{eqnarray}
\frac{1}{2}\mu^c_{<w,f>}(X)+J(w,f)+\int fw\,d\kappa-C^{-1}\int wf\,d\mu\geq 0,\ \forall\,0\leq f\in\calF_{\rm loc}.
\label{MI}
\end{eqnarray}
Then the following Hardy's inequality holds true
\begin{eqnarray}
\int f^2\,d\mu\leq C\calE[f],\ \forall\,f\in\calF.
\label{HI}
\end{eqnarray}
\label{Hardy}
\end{theo}

\begin{rk}{\rm
Condition(\ref{MI}) is fulfilled if there is a function $0<w$-q.e. $w\in\calF$ such that
\begin{eqnarray}
\calE(w,f)-C^{-1}\int wf\,d\mu\geq 0,\ \forall\,0\leq f\in\calF\cap C_0(X).
\end{eqnarray}
In particular, if $w$ is the potential of a positive measure $\mu$
charging no sets having zero capacity  and such that
$\|w\|_{\infty}<\infty$, we get
\begin{eqnarray}
\calE(w,f)=\int f\,d\mu\geq\|w\|_{\infty}^{-1}\int wf\,d\mu\
,\ \forall\,0\leq f\in\calF\cap C_0(X).
\end{eqnarray}
Obtaining therefore, the known inequality \cite{vondracek,stoll-voigt,fitz,benamor04}
\begin{eqnarray}
\int f^2\,d\mu\leq \|w\|_{\infty}\calE[f],\ \forall\,f\in\calF.
\end{eqnarray}


}
\label{discussion}
\end{rk}

We shall say that a measure $\mu$ satisfies Hardy's inequality if inequality(\ref{HI}) holds true.

\begin{proof} Without loss of generality we may and shall neglect the killing term in $\calE$.\\
Let $f$ be s.t. $wf\in\calF$. Since, by Lemma\ref{quasi-bounded}, for every compact subset $K\subset X$ there is $C_K$ s.t. $w^{-1}\leq C_K$-q.e. on $K$, we obtain by Lemma\ref{improve}, that $w^{-1}\in\calF_{b,loc}$ and $f=w^{-1}wf\in\calF_{loc}$.\\
By formula(\ref{energy}) together with the product formula(\ref{product}), we obtain
\begin{eqnarray}
\calE[wf]&=&\frac{1}{2}\mu^c_{<wf,wf>}+J[wf]\nonumber\\
&&=\frac{1}{2}\int w^2\,d\mu^c_{<f>}+\int wf\,d\mu^c_{<w,f>}+\frac{1}{2}\int f^2\,d\mu^c_{<w>}+J[wf]
\end{eqnarray}
Yielding
\begin{eqnarray*}
\calE[wf]-C^{-1}\int(wf)^2\,d\mu&=&\frac{1}{2}\int w^2\,d\mu^c_{<f>}+\int wf\,d\mu^c_{<w,f>}\nonumber\\
&&+\frac{1}{2}\int f^2\,d\mu^c_{<w>}-C^{-1}\int(wf)^2\,d\mu+J[wf]
\end{eqnarray*}
Replacing $f$ by $wf^2\in\calF_{\rm loc}$ in Eq.(\ref{MI}), we get
\begin{eqnarray}
0&&\leq\frac{1}{2}\int \,d\mu^c_{<wf^2,w>}-C^{-1}\int(wf)^2\,d\mu+J(w,wf^2)=\frac{1}{2}\int f^2\,d\mu^c_{<w>}\nonumber\\
&&+\frac{1}{2}\int w\,d\mu^c_{<f^2,w>}-C^{-1}\int(wf)^2\,d\mu+J(w,wf^2).
\end{eqnarray}
Observing that
\begin{eqnarray}
\int w\,d\mu^c_{<f^2,w>}=2\int wf\,d\mu^c_{<f,w>},
\end{eqnarray}
and that
\begin{eqnarray}
J(w,wf^2)\leq J[wf]
\end{eqnarray}
we achieve
\begin{eqnarray}
\int w^2f^2\,d\mu\leq C\calE[wf],
\end{eqnarray}
for every $f$ as given in the beginning of the proof.\\
Now let $f\in\calF\cap C_0(X)$. Then $f=ww^{-1}f$. We set $g:=w^{-1}f$. Then  $wg\in\calF$. Applying the first part of the proof and using the regularity assumption, we get the result.

\end{proof}

As an example of measures for which Hardy's inequality holds true we give
\begin{coro} Let $0<w$ be a superharmonic function and $\mu$ its Riesz charge. Then
\begin{eqnarray}
\int w^{-1}f^2\,d\mu\leq\calE[f],\ \forall\,f\in\calF.
\label{riesz charge}
\end{eqnarray}
\label{example}
\end{coro}

\begin{proof} Assume first that $w\in\calF$. Then for all $f\in\calF\cap C_c(X)$ we have
\begin{eqnarray}
\calE(w,w^{-1}f)=\int w^{-1}f\,d\mu,
\end{eqnarray}
which yields inequality(\ref{riesz charge}) by Theorem\ref{Hardy}.\\
For general $w$, let $\mu_k\uparrow\mu$, be such that $w_k:=U\mu_k\in\calF$. Then by the first step
\begin{eqnarray}
\int w^{-1}f^2\,d\mu&=&\lim_{k\to\infty}\int w^{-1}f^2\,d\mu_k\leq\lim_{k\to\infty}\int w_k^{-1}f^2\,d\mu_k\nonumber\\
&&\leq \calE[f],\ \forall\,f\in\calF,
\end{eqnarray}
which finishes the proof
\end{proof}

\begin{rk}{\rm
On the light of Corollary\ref{example}, Thereom\ref{Hardy} has the following consequence: every measure which is dominated by a constant times  the inverse of a nonnegative superharmonic  function times its Riesz charge  satisfies Hardy's inequality.\\ This result is
 exactly Fitzsimmons's result \cite{fitz}.
}
\end{rk}

\begin{exa}Improved Hardy inequality in the half-space: {\rm In this example we shall rediscover an improved Hardy inequality proved in \cite[Corollary3.1]{tidblom}. Let $d\geq 3$. Set $\R_+^d$ the upper half-space. Set
$$
\psi(x):=x_d^{\frac{1}{2}}(x_{d-1}^2+x_d^2)^{\frac{1}{4}},\ x\in \R_+^d.
$$
Let $0<\epsilon<1/4$. Then with $w:=\psi$, we get
\begin{eqnarray}
-\Delta w-\frac{1}{8}w\psi^{-2}-(\frac{1}{4}-\epsilon)wx_d^{-2}&=&
\frac{1}{4}w\big(\frac{1}{(x_{d-1}^2+x_d^2)^{2}}\nonumber\\
&&+\frac{4\epsilon}{x_d^2}
-\frac{1}{2x_d(x_{d-1}^2+x_d^2)^{\frac{1}{2}}}\big)
\geq 0
\end{eqnarray}
Thus by Theorem\ref{Hardy}, we obtain
\begin{eqnarray*}
(\frac{1}{4}-\epsilon)\int_{\R_+^d}x_d^{-2}f^2\,dx+
\frac{1}{8}\int_{\R_+^d}\frac{f^2}{x_d(x_{d-1}^2+x_d^2)^{\frac{1}{2}}}\,dx
\leq\int_{\R_+^d}|\nabla f|^2\,dx,\ \forall f\in C_0^{\infty}({\R_+^d}).
\end{eqnarray*}
Letting $\epsilon\to 0$, we derive
\begin{eqnarray}
\frac{1}{4}\int_{\R_+^d}x_d^{-2}f^2\,dx+
\frac{1}{8}\int_{\R_+^d}\frac{f^2}{x_d(x_{d-1}^2+x_d^2)^{\frac{1}{2}}}\,dx
\leq\int_{\R_+^d}|\nabla f|^2\,dx,\ \forall f\in C_0^{\infty}({\R_+^d}).
\end{eqnarray}

}
\end{exa}

As in the context of Ancona and Fitzsimmons (See \cite[Proposition 1]{ancona},\cite{fitz}) we proceed to show that a sort of converse to Theorem(\ref{Hardy}) holds true.

\begin{theo} Assume that inequality(\ref{HI}) holds true. Then for every $0<\Lam<C^{-1}$ there is $w\in\calF$, $w>0$-q.e., and fulfills condition(\ref{MI}).
\label{converse-hardy}
\end{theo}

\begin{proof} Suppose that (\ref{HI}) holds true. Then by \cite[Theorem3.1]{benamor04}, the operator
\begin{eqnarray}
K^{\mu}:=L^2(\mu)\to L^2(\mu),\ f\mapsto Uf\mu,
\end{eqnarray}
where $Uf\mu$ is the potential of $f\mu$ is bounded and $\|K^{\mu}\|\leq C$. Thus for every $0<\Lam<C^{-1}$ the operator $1-\Lam K^{\mu}$ is invertible on $L^2(\mu)$.\\
Let $\varphi\in\calF$, s.t.  $0<\varphi\leq 1$. Then there is $\psi\in L^2(\mu)$ with $\psi-\Lam K^{\mu}\psi=\varphi$-$\mu$ a.e. Thus
\begin{eqnarray}
K^{\mu}\psi-\Lam K^{\mu}(K^{\mu}\psi)=K^{\mu}\varphi-q.e.
\end{eqnarray}
Since $\varphi>0$, $K^{\mu}$ is positivity preserving and
\begin{eqnarray}
\psi=\sum_{k=0}^{\infty}\Lam^k(K^{\mu})^k\varphi,
\end{eqnarray}
we conclude that $\psi>0$-$\mu$-a.e. and $w:=K^{\mu}\psi>0$-q.e., which by Lemma\ref{quasi-bounded} yields that $w^{-1}$ is quasi-bounded.\\
For the rest of the proof, observe that for every $0\leq f\in\calF$
\begin{eqnarray}
\calE(w,f)-\Lam\int wf\,d\mu&&=\calE(K^{\mu}\psi,f)-\Lam\int wf\,d\mu\nonumber\\
&&=\calE(K^{\mu}\varphi,w)+\Lam\calE(K^{\mu}w,f)-\Lam\int wf\,d\mu
\nonumber\\&&=\int w\varphi\,d\mu\geq 0,
\end{eqnarray}
which finishes the proof.

\end{proof}

The proof of Theorem\ref{converse-hardy} shows that if the operator $1-K^{\mu}$ is invertible, then the conclusion holds true with $\Lam=C^{-1}$ as well.\\
We shall add an alternative a assumption (which is fulfilled in many cases) on the form
\begin{eqnarray}
\calE_{\mu}, D(\calE_{\mu})=\calF, \calE_{\mu}[f]=\calE[f]-\int f^2\,d\mu,
\end{eqnarray}
that ensures that the case $\Lam=C^{-1}$ is included as well.

\begin{prop} Let $\mu$ be a positive Radon measure on Borel subsets of $X$ that satisfies the Hardy's inequality with best constant $1$. Assume that there is $\Lam>0$ s.t.
\begin{eqnarray}
\int f^2\,dm\leq\Lam\calE_{\mu}[f],\ \forall\,f\in D(\calE).
\label{IHI}
\end{eqnarray}
Then for every $g\in\calF$ there is $f\in\calF$ s.t.
\begin{eqnarray}
\calE_{\mu}(\varphi,f)=\int\varphi g\,dm\, \forall\,\varphi\in\calF.
\end{eqnarray}
If in particular $g>0$-q.e. then there is $0<w$-q.e., $w\in\calF$ and satisfies condition (\ref{MI}) with $C=1$.
\end{prop}

The proof is easy, so we omit it.\\

Let $\mu$ be a positive Radon measure on Borel subsets of $X$ charging no sets having zero capacity. Assume that there is $w$ satisfying the assumptions of Theorem\ref{Hardy} with best constant $C=1$. Then Theorem\ref{Hardy} yields that  the quadratic form defined by
\begin{eqnarray}
D(\calE^{w})=\{f\in L^2(w^2dm):\,wf\in\calF\},\ \calE^{w}[f]=\calE_{\mu}[wf],
\end{eqnarray}
is a positive quadratic form. We shall prove that $\calE^w$ is,
in fact, a Dirichlet form. A proof of this result was shortly quoted
by Fitzsimmon \cite{fitz} using a probabilistic method. We shall, however prove
it using an analytical one.

\begin{prop} Under the above assumptions the form $\calE^w$ is  a Dirichlet form.
\end{prop}

\begin{proof} We develop the proof by steps.\\
{\em Step 1}: $\calE_{\mu}$ is closable. Indeed,\\
We associate to $\calE_{\mu}$ a positive symmetric operator $H_{\mu}$ such that $D(H_{\mu})=D(H)$ and
$$
(H_{\mu}f,g)=\calE_{\mu}(f,g),\ \forall\,f\in D(H_{\mu}),g\in\calF.
$$
Since $\calF$ is dense in $L^2$ then so is $D(H_{\mu})$. Thus by \cite[Theorem1.2.8]{davies}, $\calE_{\mu}$ is closable. We still denote by $\calE_{\mu}$ its closure and $H_{\mu}$ the operator associated to it via Kato's representation theorem.\\
{\em Step 2}: $\calE_{\mu}^w$ is closed. The operator $H_{\mu}^w:=w^{-1}H_{\mu}w$ is closed and for every $f,g$ s.t. $wf,wg\in D(H_{\mu}^{1/2})=\calF$ we have
\begin{eqnarray}
\big((H_{\mu}^w)^{1/2}f,g\big)_{L^2(w^2dm)}=\calE_{\mu}^w(f,g).
\end{eqnarray}
Thus $\calE_{\mu}^w$ is closed.\\
{\em step 3}: $\calE_{\mu}^w$ is a Dirichlet form. Set
\begin{eqnarray}
 \hat\calE:\ D(\hat\calE)=D(\calE_{\mu}^w),\ \hat\calE[f]=\frac{1}{2}\int w^2\,d\mu^c_{<f>}.
\end{eqnarray}
Then $\hat\calE$ is a densely defined closable positive quadratic form satisfying the truncation
property (by property(\ref{truncation})). Hence its closure is a Dirichlet form, which we still
 denote by  $\hat\calE$. We denote by $\hat H$ its related operator.\\
On the other hand we have (by Theorem\ref{hardy inequality})
\begin{eqnarray}
0\leq\hat\calE\leq\calE_{\mu}^w,
\end{eqnarray}
yielding, for every $\alp>0$
\begin{eqnarray}
0\leq (H_{\mu}^w+\alp)^{-1}\leq (\hat H+\alp)^{-1}.
\end{eqnarray}
Now since $(H_{\mu}^w+\alp)^{-1}$ is positivity preserving (because $(H_{\mu}+\alp)^{-1}$ is) and $(\hat H+\alp)^{-1}$ is Markovian, we derive that $(H_{\mu}^w+\alp)^{-1}$ is Markovian as well and $\calE_{\mu}^w$ is a Dirichlet form, wich finishes the proof.


\end{proof}

\section{Examples for strongly local Dirichlet forms}

In this section we shall concentrate on giving general and concrete examples of measures satisfying the Hardy inequality provided the Dirichlet form is strongly local. Furthermore in some positions we shall even improve Hardy's inequality.\\
These examples are mainly inspired from classical Hardy's on Euclidean domains having strong barriers \cite{ancona}.
\begin{eqnarray}
\int_{\Om}\frac{f^2(x)}{{\rm dist}(x,\partial\Om)}\,dx\leq C_{\Om}\int_{\Om}|\nabla f(x)|^2\,dx,\ \forall\,f\in W_0^1(\Om).
\label{hardy domains}
\end{eqnarray}
and from an example given by Fitzsimmons  \cite[Example4.2]{fitz}.\\

 For the sake of completeness, we recall some basic concepts related to strongly local Dirichlet forms.\\ Every strongly local Dirichlet form, $\calE$ induces a {\em pseudo-metric} on $X$ known as the {\em intrinsic metric} and defined by
\begin{eqnarray}
\rho(x,y):=\sup\big\{f(x)-f(y),\:f\in\calF_{\rm loc},\ \frac{1}{2}\mu^c_{<f>}\leq m\ {\rm on}\ X\big\},
\label{metric}
\end{eqnarray}
 where the inequality  $1/2\mu^c_{<f>}\leq m$ in the above definition means that the energy measure $\mu^c_{<f>}$ is absolutely continuous w.r.t. the reference measure $m$ with Radon-Nikodym derivative smaller that $1$.\\
Throughout this section we shall assume that {\em $\rho$ is a true metric whose topology coincides with the original one and that $(X,\rho)$ is complete}.\\
For a given  closed subset $F\subset X$, we set
\begin{eqnarray}
\rho_F(x):=\rho(x,F),\ \forall\,x\in X.
\end{eqnarray}
Then under the above assumption (See \cite[Remark after Lemma1.9]{sturm2}),
$$\rho_F\in\calF_{\rm loc}\cap C(X)\ {\rm and}\ \frac{1}{2}d\mu^c_{<\rho_F>}\leq dm.$$

Now let $\Om\subset X$ be an open fixed subset, $\calE_{\Om}$ the form defined by
\begin{eqnarray*}
\calF_{\Om}:=D(\calE_{\Om})=\big\{f\in D(\calE)\,:f=0-{\rm q.e.\ on}\ X\setminus\Om\big\}, \calE_{\Om}[f]=\calE[f].
\end{eqnarray*}
Then $\calE_{\Om}$ is a regular strongly local Dirichlet form on $L^2(\overline\Om,m)$ (\cite[Theorem4.4.3]{fuku-oshima}).\\ Set $F=X\setminus\Om$ or any closed subset of $\Om$ having zero capacity and $\calF_{\Om,{\rm loc}}$ the local domain of $\calE_{\Om}$.\\
We are in position now to extend inequality(\ref{hardy domains}) in our framework.

\begin{theo} Assume that
\begin{eqnarray}
\int\,d\mu^c_{<\rho_F,f>}\geq 0,\ \forall\,0\leq f\in\calF_{\Om,\rm loc}.
\label{positive}
\end{eqnarray}
Then
\begin{eqnarray}
\frac{1}{2}\int_{\Om}f^2\,\frac{d\mu^c_{<\rho_F>}}{\rho_F^2}\leq 4\calE[f],\ \forall\,f\in\calF_{\Om}.
\end{eqnarray}
\label{strong1}
\end{theo}
For the gradient energy form on Euclidean domains, condition (\ref{positive}) expresses the fact that $\rho_F$  is superharmonic, under which the constant $C_{\Om}$ appearing in inequality (\ref{hardy domains}) may be chosen to be equal $4$. On the light of this observation, our extension seems to be quite natural.

\begin{proof} Set $w=\rho_F^{\frac{1}{2}}$. By Theorem\ref{Hardy}, it suffices to prove
\begin{eqnarray}
\frac{1}{2}\int_{\Om}\,d\mu^c_{<w,f>}-\frac{1}{8}\int_{\Om}wf\frac{d\mu^c_{<\rho_F>}}{\rho_F^2}\geq 0,\ \forall\,0\leq f\in\calF_{\Om,{\rm loc}},
\end{eqnarray}
or equivalently
\begin{eqnarray}
\frac{1}{2}\int_{\Om}\,d\mu^c_{<w,wf>}-\frac{1}{8}\int_{\Om}w^{-2}f\,d\mu^c_{<\rho_F>}\geq 0,\ \forall\,0\leq f\in\calF_{\Om,{\rm loc}}.
\end{eqnarray}
Let $0\leq f\in\calF_{\Om,{\rm loc}}$. Owing to the product formula together with the chain rule given by Lemma\ref{improve}, a straightforward computation yields
\begin{eqnarray}
\int_{\Om}\,d\mu^c_{<w,wf>}&=&\int_{\Om}f\,d\mu^c_{<w>}+\int_{\Om}w\,d\mu^c_{<w,f>}\nonumber\\
&&=\frac{1}{4}\int_{\Om}fw^{-2}d\mu^c_{<\rho_F>}+\frac{1}{2}\int_{\Om}d\mu^c_{<\rho_F,f>},
\end{eqnarray}
obtaining thereby
\begin{eqnarray}
\frac{1}{2}\int_{\Om}\,d\mu^c_{<w,wf>}-\frac{1}{8}\int_{\Om}fw^{-2}d\mu^c_{<\rho_F>}\geq 0,
\end{eqnarray}
which completes the proof.

\end{proof}

\begin{exa}{\rm Let $\Om$ be a convex subset of the Euclidean space $\R^d$ ($d\geq 3$ if $\Om$ is unbounded) and  $\varp$ a function s.t. $\varp>0$-q.e. on $\Om$ and $\varp,\varp^{-1}\in L_{\rm loc}^2(\Om,dx)$. We define the Dirichlet form on $L^2(\overline\Om,dx)$ by
\begin{eqnarray}
\calE[f]=\int_{\Om}|\nabla f|^2\varp^2\,dx,\ \forall\,f\in C_0^{\infty}(\Om),
\end{eqnarray}
and $\calF$ being the closure of $C_0^{\infty}(\Om)$ w.r.t. $\calE_1^{\frac{1}{2}}$. Then it is known that
\begin{eqnarray}
\rho(x,y)=|x-y|,\ \forall\,x,y\in\Om.
\end{eqnarray}
Set $F=\R^d\setminus\Om$. Assume that $\varp$ satisfies condition (\ref{positive}) which reads
\begin{eqnarray}
 -\Delta\rho_F-2\varp^{-1}\nabla\varp\nabla\rho_F\geq 0.
\end{eqnarray}
(It is the case if for example $\varp=\rho_F^{-\alp},\ \alp\geq 0$). Then  conditions of Theorem\ref{strong1} are fulfilled and we get
\begin{eqnarray}
\int_{\Om}\frac{f^2}{\rho_F^2}\,dx\leq 4\int_{\Om}|\nabla f|^2\varp^2\,dx,\ \forall\,f\in\calF.
\end{eqnarray}
}
\label{exp1}
\end{exa}

Another general example is the following

\begin{theo} Let $\psi\in\calF_{\rm loc}$ be s.t. $\psi>0$-q.e., $\frac{d\mu_{<\psi>}^c}{2dm}\leq 1$ and for some constant $C>1/2$
\begin{eqnarray}
\calE(\psi,f)\leq -2C\int \psi^{-1}f\,dm,\ \forall\, 0\leq f\in\calF_{\Om,\rm loc}.
\label{dominated energy}
\end{eqnarray}
Set $\beta:=C-\frac{1}{2}$. Then
\begin{eqnarray}
\int f^2\psi^{-2}\,dm\leq\beta^{-2}\calE[f],\ \forall\,f\in\calF_{\Om}.
\end{eqnarray}
\label{strong2}
\end{theo}

\begin{proof} Let $0\leq f\in\calF_{\Om,\rm loc}$. Changing $f$ by $\psi^{-2\beta-1}f$ in inequality(\ref{dominated energy}) and applying the chain rule we achieve
\begin{eqnarray}
\frac{1}{2}\int\psi^{-2\beta-1}\,d\mu^c_{<\psi,f>}-\frac{1}{2}(2\beta+1)
\int\psi^{-2\beta-2}f\,d\mu^c_{<\psi>}\leq -2C\int\psi^{-2\beta-2}f\,dm.
\end{eqnarray}
Using the latter inequality together with the assumption $\frac{d\mu_{<\psi>}^c}{2dm}\leq 1$, we obtain
\begin{eqnarray}
\frac{1}{2}\int\,d\mu^c_{<\psi^{-\beta},\psi^{-\beta}f>}&&=
-\frac{\beta}{2}\int\psi^{-2\beta-1}\,d\mu^c_{<\psi,f>}
+\frac{\beta^2}{2}\int\psi^{-2\beta-2}f\,d\mu^c_{<\psi>}\nonumber\\
&&\geq 2C\beta\int\psi^{-2\beta-2}f\,dm
-\frac{\beta}{2}(2\beta+1)\int\psi^{-2\beta-2}\,d\mu^c_{<\psi>}\nonumber\\
&&+\frac{\beta^2}{2}\int\psi^{-2\beta-2}f\,d\mu^c_{<\psi>}\nonumber\\
&&\geq 2(\beta+\frac{1}{2})\beta\int\psi^{-2\beta-2}f\,dm
-\beta(\beta+1)\int\psi^{-2\beta-2}f\,dm.
\end{eqnarray}
Thus
\begin{eqnarray}
\frac{1}{2}\int\,d\mu^c_{<\psi^{-\beta},\psi^{-\beta}f>}-\beta^2
\int\psi^{-2\beta-2}f\,dm\geq 0,
\end{eqnarray}
and $w=\psi^{-\beta}$ satisfies condition(\ref{MI}), with $d\mu=\psi^{-2}dm$, which completes the proof.
\end{proof}

\begin{exa}{\rm We take an other time the Dirichlet form of Example\ref{exp1}, with $d\geq 3$. We suppose that $\Om$ is star-shaped around one of its points $x_0\in\Om$. We choose $F=\{x_0\}$ and assume that points have zero capacity. Then
$$
\rho(x):=\rho_F(x)=|x-x_0|.
$$
We choose $\psi(x)=\rho(x)$ and $\varp(x)=e^{\rho(x)}$. Then $\frac{d\mu_{<\psi>}^c}{2dm}\leq 1$.\\
On the other hand condition (\ref{dominated energy2}) reads
\begin{eqnarray}
d-1+2\rho(x)\geq 2Ce^{-2\rho(x)},
\end{eqnarray}
which is fulfilled with $C=\frac{d-1}{2}$. Thus we get
\begin{eqnarray}
\int_{\Om}\frac{f(x)^2}{|x-x_0|^2}e^{2|x-x_0|}\,dx\leq (\frac{d-2}{2})^{-2}\int_{\Om}|\nabla f(x)|^2e^{2|x-x_0|}\,dx,\ \forall\,f\in\calF.
\end{eqnarray}
}
\end{exa}

\begin{exa} {\rm We investigate in this example the Dirichlet form given by: Set $\sigma(x)=(1+|x|^2)^{\frac{1}{2}}$ and
\begin{eqnarray}
\calE[f]=\int_{\R^d}|\nabla f|^2\,dx+\int_{\R^d}f^2\sigma^{\lam}(x)\,dx,
\end{eqnarray}
considered on the space $L^2(\R^d,\sigma^{\lam}dx)$. In this situation the intrinsic metric is given by \cite{cipriani-grillo}
\begin{eqnarray}
\rho(0,x)=\ln(|x|+\sqrt{1+|x|^2}).
\end{eqnarray}
We set $\psi(x):=\rho(0,x),\ \forall\,x\in\R^d$ and suppose that $d\geq 3$.\\ From the property of the intrinsic metric we derive $\frac{d\mu_{<\psi>}^c}{2dm}\leq 1$. The second condition imposed on $\psi$ reads
\begin{eqnarray}
-\Delta\psi+\psi\sigma^{\lam}\leq -C\psi^{-1}\sigma^{\lam},
\end{eqnarray}
or equivalently
\begin{eqnarray}
\frac{d-1}{|x|}+\frac{d}{(1+|x|^2)^{3/2}}
&-&\ln(|x|+\sqrt{1+|x|^2})\sigma^{\lam}(x)\nonumber\\
&&\geq C\frac{\sigma^{\lam}(x)}{\ln(|x|+\sqrt{1+|x|^2})},\ \forall\,x\in\R^d\setminus\{0\},
\end{eqnarray}
with $C>1/2$. Obviously this condition can not be fulfilled if $\lam\geq 0$. However if $\lam<0$ and $-\lam$ is big enough then the latter condition is satisfied and we obtain for such $\lam$
\begin{eqnarray*}
\int_{\R^d} f^2(x)\ln^{-2}(|x|+\sqrt{1+|x|^2})\,dx\leq\beta_{\lam}^{-2}\big(\int_{\R^d}|\nabla f|^2\,dx+\int_{\R^d}f^2\sigma^{\lam}(x)\,dx\big),\ \forall\,f\in D(\calE).
\end{eqnarray*}

}
\end{exa}

The latter theorem may be improved in the following way

\begin{theo} Let $\psi\in\calF_{\rm loc}$ be s.t. $\psi>0$-q.e. and for some constant $C>1/2$
\begin{eqnarray}
\calE(\psi,f)\leq -C\int \psi^{-1}f\,d\mu_{<\psi>}^c,\ \forall\, 0\leq f\in\calF_{\Om,\rm loc}.
\label{dominated energy2}
\end{eqnarray}
Set $\beta:=C-\frac{1}{2}$. Then
\begin{eqnarray}
\frac{1}{2}\int f^2\psi^{-2}\,d\mu_{<\psi>}^c\leq\beta^{-2}\calE[f],\ \forall\,f\in\calF_{\Om}.
\end{eqnarray}
\label{strong3}
\end{theo}

The proof runs as the previous one so we omit it.\\

\begin{rk}{\rm Inequality(\ref{dominated energy2}) is fulfilled with $C=1$ if
\begin{eqnarray}
\calE(\log\psi,f)\leq 0,\ \forall\, 0\leq f\in\calF_{\Om,\rm loc}.
\end{eqnarray}
}
\end{rk}

On the light of Theorems\ref{strong2}-\ref{strong3} and being inspired by a result due Filippas-Moschini-Tertikas \cite[Theorem3.2]{filippas},
 we shall improve, in some respect, the Hardy  inequality.

\begin{theo} Assume that conditions imposed on $\rho_F$ in Theorem\ref{strong1} and on $\psi$ in Theorem\ref{strong3}
are fulfilled. Then the following improved Hardy's
inequality
\begin{eqnarray}
\frac{1}{2}\int_{\Om}f^2\,\frac{d\mu^c_{<\rho_F>}}{\rho_F^2}\leq 4\big(\calE[f]
-\frac{\beta^2}{2}\int f^2\psi^{-2}\,d\mu_{<\psi>}^c\big),\ \forall\,f\in \calF_{\Om},
\label{ImHI}
\end{eqnarray}
holds true, provided
\begin{eqnarray}
\int_{\Om}\psi^{-2\beta}\,d\mu_{<\rho_F,f>}^c\geq 0\,\ \forall\,0\leq f\in \calF_{\Om,{\rm loc}}.
\label{joint condition}
\end{eqnarray}
\label{IHI1}
\end{theo}

\begin{proof} Set $w_1=\psi^{-\beta}$, $w_2=\rho_F^{\frac{1}{2}}$ and $f=w_1w_2g\in\calF_{\Om}$. Then
\begin{eqnarray}
\calE[w_1w_2g]&=&\frac{1}{2}\int\,d\mu_{<w_1w_2g>}^c=\frac{1}{2}\int(w_1w_2)^2\,d\mu_{<g>}^c
+\frac{1}{2}\int\psi^{-2\beta}g\,d\mu_{<g,\rho_F>}^c\nonumber\\
&&-\beta\int\psi^{-(2\beta+1)}\rho_Fg\,d\mu_{<g,\psi>}^c
-\frac{\beta}{2}\int\psi^{-(2\beta+1)}g^2\,d\mu_{<\rho_F,\psi>}^c
\nonumber\\
&&+\frac{1}{8}\int\psi^{-2\beta}g^2\rho_F^{-1}\,d\mu_{<\rho_F>}+
\frac{\beta^2}{2}\int\psi^{-2(\beta+1)}g^2\rho_F\,d\mu_{<\psi>}^c.
\end{eqnarray}

Yielding

\begin{eqnarray}
&\calE[w_1w_2g]&-\frac{\beta^2}{2}\int(w_1w_2)^2\psi^{-2}g^2\,d\mu_{<\psi>}^c
-\frac{1}{8}\int(w_1w_2)^2\rho_F^{-2}g^2\,d\mu_{<\rho_F>}^c=\nonumber\\
&& \frac{1}{2}\int(w_1w_2)^2\,d\mu_{<g>}^c
+\frac{1}{2}\int\psi^{-2\beta}g\,d\mu_{<g,\rho_F>}^c\nonumber\\
&&-\frac{\beta}{2}\int\psi^{-(2\beta+1)}\rho_F\,d\mu_{<g^2,\psi>}^c
-\frac{\beta}{2}\int\psi^{-(2\beta+1)}g^2\,d\mu_{<\rho_F,\psi>}^c
\end{eqnarray}

Observe that by assumptions the first two integrals in the latter equality are positive.
We shall prove that the remainder which we denote by $R$ is positive as well. We rewrite $R$
with the help of the product formula
\begin{eqnarray}
R=-\frac{\beta}{2}\int\psi^{-(2\beta+1)}\,d\mu_{<g^2\rho_F,\psi>}^c.
\end{eqnarray}
Owing to inequality (\ref{dominated energy2}), we achieve
\begin{eqnarray}
R&\geq& C\beta\int\psi^{-2\beta-2}g^2\rho_F\,d\mu_{<\psi>}^c
-\beta(\beta+1/2)\int\psi^{-2\beta-2}g^2\rho_F\,d\mu_{<\psi>}^c=0,
\end{eqnarray}

which was to be proved.

\end{proof}

We illustrate the improved Hardy's inequality by an example.

\begin{exa}{\rm We reconsider the Dirichlet form of Example\ref{exp1}.
We suppose that $\Om=B_R$, the open Euclidean ball centered at $0$ with radius $R>0$. We set
$$
\rho(x):=R-|x|,\ x\in B_R.
$$
We fix $\alp\in[0,1/2]$ and choose
\begin{eqnarray}
\varp(x)=\rho(x)^{-\alp}\ {\rm and}\ \psi(x)=|x|,\ \forall\,x\in B_R.
\end{eqnarray}
Then condition (\ref{dominated energy2}) imposed on $\psi$ reads
\begin{eqnarray}
\frac{1-d}{|x|}+\frac{2}{|x|}\varp^{-1}x\cdot\nabla\varp\leq -C\frac{1}{|x|}\varp^2,
\end{eqnarray}
which is always satisfied. However the condition $C>1/2$ is fulfilled if and only if
$$
(d-1)R^{2\alp}>1.
$$
Whence from now on we assume in this example that $d>1$ and $R$ satisfies the latter condition (big $R$).\\
The condition imposed on $\rho$ reads
\begin{eqnarray}
\frac{-1+d}{|x|}+2\alp\rho^{-1}(x)\geq 0, {\rm on}\ B_R,
\end{eqnarray}
which is always true.\\
Lastly the condition \ref{joint condition} imposed jointly on $\psi$ and $\rho$ reads
\begin{eqnarray}
-{\rm div}\big(\psi^{-2\beta}\rho^{-2\alp}\nabla\rho)\geq 0,
\end{eqnarray}
or equivalently
\begin{eqnarray}
\frac{2\beta}{|x|}+\frac{2\alp}{R-|x|}-\Delta\rho\geq 0,
\end{eqnarray}
which is always fulfilled.\\
Thus we get, with $\beta:=(d-1)R^{2\alp}-1/2$, for every $f\in W^1_0(B _R)$
\begin{eqnarray*}
\int_{B_R}\frac{f^2}{(R-|x|)^2}\,dx\leq
4\big( \int_{B_R}|\nabla{f}|^2(R-|x|)^{-2\alp}\,dx-
\beta^2\int_{B_R}\frac{f^2}{|x|^2}(R-|x|)^{-2\alp}\,dx\big).
\end{eqnarray*}

}
\end{exa}

Other conditions may also lead to an improved Hardy's inequality. Indeed, following the lines of the latter proof one get
\begin{prop} Assume that $\rho_F$ satisfies conditions of Theorem\ref{strong1}, that
$$\frac{1}{2}\frac{d\mu_{<\rho_F>}^c}{dm}=1,$$
and that $\psi$ satisfies conditions of Theorem\ref{strong2}. Then

\begin{eqnarray}
\int_{\Om}\frac{f^2}{\rho_F^2}\,dm\leq 4\big(\calE[f]
-\frac{\beta^2}{2}\int f^2\psi^{-2}\,d\mu_{<\psi>}^c\big),\ \forall\,f\in \calF_{\Om},
\label{ImHI2}
\end{eqnarray}
holds true, provided
\begin{eqnarray}
\int_{\Om}\psi^{-2\beta}\,d\mu_{<\rho_F,f>}^c\geq 0\,\ \forall\,0\leq f\in \calF_{\Om,{\rm loc}}.
\label{joint condition'}
\end{eqnarray}
\label{IHI2}

\end{prop}

Set ${\rm Cap}_{\Om}$, the capacity induced by $\calE_{\Om}$. In conjunction  with the equivalence between isocapacitary inequality and Hardy's inequality \cite{fitz,benamor-osaka} the latter proposition leads to the following lower estimate for the capacity of compact sets
\begin{eqnarray}
\int_{K}\frac{1}{\rho_F^2}\,dm+\frac{\beta^2}{2}\int_K \psi^{-2}\,d\mu_{<\psi>}^c\leq 4{\rm Cap}_{\Om}(K),\ \forall K\subset\Om,\ {\rm compact}.
\end{eqnarray}


\bibliography{bib-hardy}

\begin{thebibliography}{CMR94}

\bibitem[Ada73]{adams73}
David~R. Adams.
\newblock A trace inequality for generalized potentials.
\newblock {\em Studia Math.}, 48:99--105, 1973.

\bibitem[AH96]{adams}
David~R. Adams and Lars~Inge Hedberg.
\newblock {\em Function spaces and potential theory}.
\newblock Springer-Verlag, Berlin, 1996.

\bibitem[Anc86]{ancona}
Alano Ancona.
\newblock On strong barriers and an inequality of {H}ardy for domains in
  $\mathbb{\R}^n$.
\newblock {\em J.London Math.Soc.}, 2(34):274--290, 1986.

\bibitem[BA04]{benamor04}
Ali Ben~Amor.
\newblock Trace inequalities for operators associated to regular {D}irichlet
  forms.
\newblock {\em Forum Math.}, 16(3):417--429, 2004.

\bibitem[BA05]{benamor-osaka}
Ali Ben~Amor.
\newblock On the equivalence between trace and capacitary inequalities for the
  abstract space of {B}essel potentials.
\newblock {\em Osaka J.Math.}, 42:11--26, 2005.

\bibitem[CG98]{cipriani-grillo}
Fabio Cipriani and Gabriele Grillo.
\newblock ${L^p}$-exponentiel decay for solutions to functional equations in
  local {D}irichlet spaces.
\newblock {\em J.reine angew.Math.}, 496:163--179, 1998.

\bibitem[CMR94]{roeckner-homeo}
Z.Q. Chen, Zhi-Ming Ma, and Michael R\"ockner.
\newblock Quasi-homeomorphism of {D}irichlet forms.
\newblock {\em Nagoya Math.J.}, 136:1--15, 1994.

\bibitem[Dav89]{davies}
Eduard~B. Davies.
\newblock {\em Heat kernels and spectral theory}.
\newblock Cambridge University Press, Cambridge, 1989.

\bibitem[Fit00]{fitz}
P.J. Fitzsimmons.
\newblock {Hardy's inequality for Dirichlet forms.}
\newblock {\em J. Math. Anal. Appl.}, 250(2):548--560, 2000.

\bibitem[FLA07]{filippas}
Statis Filippas, Moschini Luisa, and Tertikas Achilles.
\newblock Sharp two-sided heat kernel estimates for critical {S}chr\"odinger
  operators on bounded domains.
\newblock {\em Comm.Math.Phys.}, 273:237--281, 2007.

\bibitem[F{\=O}T94]{fuku-oshima}
Masatoshi Fukushima, Y{\=o}ichi {\=O}shima, and Masayoshi Takeda.
\newblock {\em Dirichlet forms and symmetric {M}arkov processes}.
\newblock Walter de Gruyter \& Co., Berlin, 1994.

\bibitem[FU03]{fuku02}
Masatoshi Fukushima and Toshihiro Uemura.
\newblock Capacitary bounds of measures and ultracontractivity of time changed
  processes.
\newblock {\em J.Math.Pures Appl.}, 82:553--572, 2003.

\bibitem[Kai92]{kaimanovich}
Vadim~A. Kaimanovich.
\newblock Dirichlet norms, capacities and generalized isoperimetric
  inequalities for {M}arkov operators.
\newblock {\em Potential Anal.}, 1(1):61--82, 1992.

\bibitem[Kuw98]{kuwae}
Kazuhiro Kuwae.
\newblock {F}unctional calculus for {D}irichlet forms.
\newblock {\em Osaka J.Math.}, 35:683--715, 1998.

\bibitem[Maz85]{mazja85}
Vladimir~G. Maz'ja.
\newblock {\em Sobolev spaces}.
\newblock Springer-Verlag, Berlin, 1985.

\bibitem[Rv06]{rao}
Murali Rao and Hrvoje \v{S}iki\'c.
\newblock {P}otential-theoretic nature of {H}ardy's inequality for {D}irichlet
  forms.
\newblock {\em J.Math.Anal.Appl.}, 318:781--786, 2006.

\bibitem[Stu95]{sturm2}
Karl~T. Sturm.
\newblock Analysis on local {D}irichlet spaces-{II} . {Upper} {G}aussian
  estimates for the fundamental solutions of parabolic equations .
\newblock {\em Osaka J.Math.}, 32:257--312, 1995.

\bibitem[SV96]{stoll-voigt}
Peter Stollmann and J\"urgen Voigt.
\newblock Perturbation of dirichlet forms by measures.
\newblock {\em Potential Anal.}, 5:109--138, 1996.

\bibitem[Tid05]{tidblom}
Jesper Tidblom.
\newblock A {H}ardy inequality in the half-space.
\newblock {\em J.Funct.Anal.}, 221:482--495, 2005.

\bibitem[Von96]{vondracek}
Zoran Vondracek.
\newblock {An estimate for the $L\sp 2$-norm of a quasi continuous function
  with respect to a smooth measure.}
\newblock {\em Arch. Math.}, 67(5):408--414, 1996.

\end{thebibliography}

\end{document}